\documentclass[a4paper,11pt]{article}
\usepackage[T1]{fontenc}
\usepackage{lmodern,amsmath,amsfonts,amssymb,graphicx,float,microtype,underscore,mathtools,xurl,subcaption}
\usepackage[dvipsnames,svgnames,table]{xcolor}
\usepackage[labelsep = period, justification = centering]{caption}
\usepackage[shortlabels]{enumitem}
\setlist[itemize]{topsep=0ex,itemsep=0ex,parsep=0ex}
\setlist[enumerate]{topsep=0ex,itemsep=0ex,parsep=0ex}
\usepackage[unicode=true]{hyperref}
\hypersetup{
colorlinks,
breaklinks=true,
linkcolor={blue!60!black},
citecolor={black},
urlcolor={blue!60!black},
pdftitle={The Erdős--Sós Theorem}}
\usepackage[capitalise, compress, nameinlink, noabbrev]{cleveref}
\usepackage[longnamesfirst,numbers,sort&compress]{natbib}
\makeatletter
\def\NAT@spacechar{~}
\makeatother
\usepackage[tmargin=30mm,bmargin=30mm,lmargin=30mm,rmargin=30mm]{geometry}
\renewcommand{\baselinestretch}{1.1}
\renewcommand{\epsilon}{\varepsilon}
\renewcommand{\emptyset}{\varnothing}

\renewcommand{\geq}{\geqslant}
\renewcommand{\leq}{\leqslant}
\renewcommand{\thefootnote}{\fnsymbol{footnote}}
\usepackage{amsthm,thmtools}

\declaretheoremstyle[
spaceabove = .2\baselineskip plus .2\baselineskip minus .2\baselineskip, 
spacebelow = .5\baselineskip plus .2\baselineskip minus .2\baselineskip,
headfont = \normalfont\itshape,
notefont = \mdseries, 
notebraces = {}{},
bodyfont = \normalfont,
postheadspace = .5em,
headpunct = .,
qed = \qedsymbol
]{proofstyle}

\declaretheorem[name = Theorem, style = plain]{thm}

\declaretheorem[name = Lemma, numberlike = thm, style = plain]{lem}

\crefname{thm}{Theorem}{Theorems}
\crefname{lem}{Lemma}{Lemmas}
\crefname{obs}{Observation}{Observations}
\crefname{conj}{Conjecture}{Conjectures}
\crefname{claim}{Claim}{Claims}
\crefname{prob}{Problem}{Problems}
\crefname{prop}{Proposition}{Propositions}
\crefname{cor}{Corollary}{Corollaries}
\crefname{rmk}{Remark}{Remarks}
\crefname{appendix}{Appendix}{Appendices}
\begin{document}
\title{\bfseries\fontsize{18pt}{20pt}\selectfont The Erdős--Sós Theorem}

\author{David~R.~Wood\,\footnotemark[2]}

\maketitle

\begin{abstract}
We present an exposition of a proof, discovered by GPT-6 Astra, of the Erdős--Sós Conjecture, which states that every graph with average degree greater than $t-2$ contains every tree on $t\geq 2$ vertices.
\end{abstract}

\footnotetext[2]{School of Mathematics, Monash University, Melbourne, Australia (\textsf{\href{mailto:david.wood@monash.edu}{david.wood@monash.edu}}). Research supported by the Australian Research Council and by NSERC. }

\renewcommand{\thefootnote}{\arabic{footnote}}
\section{Introduction}

The following result was conjectured by Paul Erdős and Vera Sós at a conference in Smolenice in 1963 \citep{Erdos64} as a natural generalisation of the earlier Erdős--Gallai Theorem for paths~\citep{ErdosGallai59}\footnote{We consider finite undirected graphs $G$ with vertex-set $V(G)$ and edge-set $E(G)$. A graph $G$ \emph{contains} a graph $H$ if a subgraph of $G$ is isomorphic to $H$.}. The conjecture was recently proved by GPT-6 Astra \citep{ErdosSosSolution} as part of the \textsc{FrontierMath Erdős} project \citep{FrontierMathErdos} and 
formally verified in Lean~\citep{Adamczewski26}.

\begin{thm}
\label{ErdosSosTheorem}
For every integer $t\geq 2$, every graph with average degree greater than $t-2$ contains every tree with $t$ vertices. 
\end{thm}

This paper provides an exposition of the proof of \cref{ErdosSosTheorem} and some background on the conjecture. 

First note that a greedy embedding algorithm shows that for every integer $t\geq 2$, every graph with \emph{minimum} degree at least $t-1$ contains every tree with $t$ vertices.  Every graph with average degree at least $2t-2$ contains a subgraph with minimum degree at least $t-1$.   So every graph with average degree at least $2t-2$ contains every tree with $t$ vertices.  The point of \cref{ErdosSosTheorem} is to improve this average degree bound to $>t-2$, which is best possible, since the complete graph on $t-1$ vertices has average degree $t-2$ and contains no tree on $t$ vertices. 

In the 1990s, Ajtai, Koml\'os, Simonovits and Szemer\'edi announced a proof of the Erdős--Sós Conjecture for all sufficiently large $t$ (see~\citep{AKSS15}). The proof, based on the Szemer\'edi Regularity Lemma, is long and complicated, and has never been written in full (as far as I am aware). The Erd\H{o}s--S\'os Conjecture came to be recognised by the graph theory community as an important and difficult problem. Indeed, \citet[Section~3.5.1]{CG98} described it as ``one of the most tantalizing problems in extremal graph theory''. As a reflection of this importance, many partial results were obtained, including verifications of the conjecture 
for various families of trees \citep{Pokrovskiy24,Rozhon19,McLennan05,FS07,Sun07,Fan13,Sidorenko89,BPSS21}, 
for various families of host graphs \citep{ReedStein26b,WLL00,Haxell01a,BD96,SW97,ET10,ET13,YL04}, and
for various restrictions on the relative sizes of the tree and the host graph~\citep{Zhou84,Tiner10,Wozniak96,ReedStein26a,GZ16,ReedStein25,STY85}. 
Approximate variants of the Erd\H{o}s--S\'os Conjecture and related degree-condition conjectures have also been widely studied~\citep{DPRS26,DPRS23,BPS20,HRSW20,BPS19,HPSS15,HKPSSS-I,HKPSSS-II,HKPSSS-III,HKPSSS-IV,PiguetStein12}; see \citep{Stein20} for a survey.
 
%{L}oebl--{K}oml{\'o}s--{S}{\'o}s conjecture 1995: Every $n$-vertex graph with at least $n/2$ vertices of degree more than $t-2$ contains every tree on $t$ vertices   \citep{HPSS15,HKPSSS-I,HKPSSS-II,HKPSSS-III,HKPSSS-IV,PiguetStein12}.

% Brown–Erdős–Sós conjecture \citep{BES73}

We finish this introduction with a known \citep{Stein20,ReedStein26a,BPS21} application of the Erdős--Sós Theorem to multicolour Ramsey numbers. For a graph $G$ and integer $k\geq 1$, let $R(G;k)$ be the minimum integer $n$ such that for every edge-$k$-colouring of the complete graph $K_n$ some monochromatic subgraph of $K_n$ is isomorphic to $G$. \citet{ErdosGraham75} showed that $R(T; k) \geq k(t-2) + 2$ for any tree $T$ on $t\geq 2$ vertices and for sufficiently large $k\equiv 1\pmod{t}$. They asked whether $R(T; k) \leq kt + O(1)$. 

\begin{thm}
For every tree $T$ on $t\geq 2$ vertices, and for every integer $k\geq 1$, 
\[R(T,k) \leq k(t-2)+2.\]
\end{thm}

\begin{proof}
Consider any edge-$k$-colouring of $K_n$ where $n=k(t-2)+2$. 
If there is no monochromatic copy of $T$, then each colour occurs at most $(t-2)n/2$ times by \cref{ErdosSosTheorem}, 
implying $\binom{n}{2}\leq k(t-2)n/2$ and $n \leq k(t-2)+1$. Thus there is a monochromatic copy of $T$.
\end{proof}

 %%%%%%%%%%
\section{The Proof}

Let $G$ be a graph with $n$ vertices and $m$ edges. 
Let $L(G)$ be the set of all linear orderings of $V(G)$. 
Let $M(G)$ be the set of all pairs $(\pi,i)$ where $\pi=(v_1,\dots,v_n) \in L(G)$ and $v_1v_i\in E(G)$ (implying $i\in\{2,\dots,n\}$). 
For each edge $vw\in E(G)$ there are $2(n-1)!$  pairs $(\pi,i)\in M(G)$ with $\{v,w\}=\{v_1,v_i\}$. 
So 
\begin{equation}
\label{MainEquality}
|M(G)|= 2m(n-1)!
\end{equation}
Let $T$ be a tree with $t\geq 2$ vertices. 
For  $r\in V(T)$, let $R(T,r)$ be the set of all pairs $(\pi,i)\in M(G)$ such that $G[\{v_1,\dots,v_i\}]$ contains $T$ with $r$ mapped to $v_1$. 
Below we prove
\begin{equation}
\label{MainInequality}
|M(G)| \leq |R(T,r)| + (t-2)n!.
\end{equation}
Together, \eqref{MainEquality} and \eqref{MainInequality} imply $2m(n-1)! \leq |R(T,r)| + (t-2)n!$. If $G$ does not contain $T$, then $R(T,r)=\emptyset$ for any $r\in V(T)$, implying
$2m \leq (t-2)n$; that is, $G$ has average degree $\frac{2m}{n}\leq t-2$. 
Equivalently, if $G$ has average degree greater than $t-2$, then $G$ contains $T$. 
This proves \cref{ErdosSosTheorem}.

It remains to prove \eqref{MainInequality}. We use the next two lemmas. 

\begin{lem}
\label{Leaf}
For any graph $G$ with $n$ vertices, for any tree $T$ with at least three vertices, if $r$ is a leaf of $T$ and $r'$ is the neighbour of $r$ in $T$, then 
\[ |R(T-r,r') | \leq |R(T,r)| + n!.\]
\end{lem}

\begin{proof}
Let $X$ be the set of all pairs $(\pi,j)\in R(T-r,r')$ such that $j$ is the minimum integer $i$ (if any exist) such that  $(\pi,i)\in R(T-r,r')$. 
So $|X|\leq |L(G)|=n!$. 

Consider $(\pi,i)  \in R(T-r,r')\setminus X$, where $\pi = (v_1,\dots,v_n)$. 
Thus $(\pi,j)\in R(T-r,r')$ for some $j\in\{2,\dots,i-1\}$. 
Define 
\[f((\pi,i)) := (( v_i,v_{i-1}\,\dots,v_1,v_{i+1},v_{i+2},\dots,v_n),i).\]
In this ordering, the first vertex (namely $v_i$) is adjacent to the $i$-th vertex (namely $v_1$), so $f((\pi,i)) \in M(G)$. 
By construction, $G[\{v_1,\dots,v_j\}]$ contains $T-r$ with $r'$ mapped to $v_1$, and $j\leq i-1$. 
Thus $G[\{v_1,\dots,v_i\}]$ contains $T$ with $r$ mapped to $v_i$. 
Hence $f((\pi,i))\in R(T,r)$. 

We have shown that  $f$ is a function from $R(T-r,r')\setminus X$ to $R(T,r)$. 
We claim that $f$ is injective. 
Suppose that $f((\pi,i))=f((\pi',i'))$ for some $(\pi,i),(\pi',i')\in R(T-r,r')\setminus X$. 
Say $\pi=(v_1,\dots,v_n)$ and $\pi'=(w_1,\dots,w_n)$. 
So $(\pi,j) \in X$ and $ (\pi',j') \in X$ for some $j\in\{2,\dots,i-1\}$ and $j'\in\{2,\dots,i'-1\}$. 
By construction,
\begin{align*}
f((\pi,i)) & = (( v_i,v_{i-1}\,\dots,v_1,v_{i+1},v_{i+2},\dots,v_n),i)\\
= f((\pi',i')) & = (( w_{i'},w_{i'-1}\,\dots,w_1,w_{i'+1},w_{i'+2},\dots,w_n),i').
\end{align*}
Thus $i=i'$ and 
$v_i,v_{i-1}\,\dots,v_1=w_{i'},w_{i'-1}\,\dots,w_1$, implying
$v_{i+1},v_{i+2},\dots,v_n=\linebreak w_{i'+1},w_{i'+2},\dots,w_n$.
Together this shows that $(\pi,i) = (\pi',i')$. 
Therefore $f$ is an injection from $R(T-r,r')\setminus X$ to $R(T,r)$. 

Thus $|R(T-r,r')| - |X|=|R(T-r,r')\setminus X| \leq |R(T,r)|$. 
Hence $|R(T-r,r')|  \leq |X| + |R(T,r)| \leq n! + |R(T,r)|$, as desired. 
\end{proof}

\begin{lem}
\label{NonLeaf}
For any graph $G$ with $n$ vertices, for any tree $T$ with at least three vertices, for any non-leaf vertex $r\in V(T)$, for any subtrees $T_1$ and $T_2$ of $T$ such that $T_1\cup T_2=T$ and $V(T_1\cap T_2)=\{r\}$, 
\[ |R(T_1,r)| + |R(T_2,r)| \leq |M(G)| + |R(T,r)| + n!. \]
\end{lem}

\begin{proof}
Let $X$ be the set of all pairs $(\pi,j)\in R(T_1,r) \setminus R(T,r)$ such that $j$ is the minimum integer $i$ (if any exist) such that  $(\pi,i)\in R(T_1,r) \setminus R(T,r)$. 
For each $\pi\in L(G)$ there is at most one pair $(\pi,j)$ in $X$. 
So $|X|\leq |L(G)|=n!$. 

Consider $(\pi,i) \in ( R(T_1,r) \setminus R(T,r) ) \setminus X$, where $\pi=(v_1,\dots,v_n)$. 
Thus $(\pi,j)\in X$ for some $j\in\{2,\dots,i-1\}$. 
Define
\[f((\pi,i)) := (( v_1,v_{j+1},\dots,v_i, v_2,\dots,v_j, v_{i+1},\dots,v_n),i-j+1).\] 
Since $j\leq i-1$ we have $i-j+1\geq 2$. 
In this ordering, the first vertex (namely $v_1$) is adjacent to the $(i-j+1)$-th vertex (namely $v_i$), so $f((\pi,i)) \in M(G)$. 

Suppose that $f((\pi,i))\in R(T_2,r)$. 
Then $G[\{v_1,v_{j+1},\dots,v_i\}]$ contains $T_2$ with $r$ mapped to $v_1$. 
Since $G[\{v_1,v_2,\dots,v_j\}]$ contains $T_1$ with $r$ mapped to $v_1$,  
this implies that $G[\{v_1,\dots,v_i\}]$ contains $T$ with $r$ mapped to $v_1$. 
Hence  $(\pi,i)\in R(T,r)$, which is a contradiction. 
So $f((\pi,i)) \not\in R(T_2,r)$. 

We have shown that  $f$ is a function from $( R(T_1,r) \setminus R(T,r) ) \setminus X$ to $M(G) \setminus R(T_2,r)$. 
We claim that $f$ is injective. 
Suppose that $f((\pi,i))=f((\pi',i'))$ for some $(\pi,i),(\pi',i')\in (R(T_1,r) \setminus R(T,r) ) \setminus X$. 
Say $\pi=(v_1,\dots,v_n)$ and $\pi'=(w_1,\dots,w_n)$. 
So $(\pi,j) \in X$ and $ (\pi',j') \in X$ for some $j\in\{2,\dots,i-1\}$ and $j'\in\{2,\dots,i'-1\}$. 
By construction,
\begin{align*}
f((\pi,i)) & = (( v_1,v_{j+1},\dots,v_i, v_2,\dots,v_j,v_{i+1},\dots,v_n),i-j+1)\\
= f((\pi',i')) & = (( w_1,w_{j'+1},\dots,w_{i'}, w_2,\dots,w_{j'}, w_{i'+1},\dots,w_n),i'-j'+1).
\end{align*}
Thus $i-j=i'-j'$ and 
$v_1,v_{j+1},\dots,v_i=w_1,w_{j'+1},\dots,w_{i'}$.
Without loss of generality, $j \leq j'$. 
Then $v_1,v_2,\dots,v_j=w_1,w_2,\dots,w_{j}$. 
Since $(\pi,j)\in X$, we have $(\pi',j)\in X$. 
By the minimality of $j'$, we have $j=j'$. 
Thus $i=i'$ and $v_2,\dots,v_j,v_{i+1},\dots,v_n=w_2,\dots,w_{j'}, w_{i'+1},\dots,w_n$. 
Together this shows that $(\pi,i)=(\pi',i')$. 
Therefore $f$ is injective. 

Thus $|( R(T_1,r) \setminus R(T,r) ) \setminus X | \leq |M(G) \setminus R(T_2,r)|$. 
By construction, $R(T,r) \subseteq R(T_1,r)$ and $R(T_2,r) \subseteq M(G)$. 
Thus
\[ | R(T_1,r)|  - | R(T,r) |  - |X | = 
|( R(T_1,r) \setminus R(T,r) ) \setminus X | \leq |M(G) \setminus R(T_2,r)| = |M(G)| - | R(T_2,r)| ,\] 
and 
$|R(T_1,r)| +  | R(T_2,r)|  \leq |M(G)| + | R(T,r) |+ |X | \leq |M(G)| +|R(T,r)| + n!$, as desired. 
\end{proof}

\begin{proof}[Proof of \eqref{MainInequality}]
We proceed by induction on $t\geq2$. 
In the base case,  $t=2$, for any $t$-vertex tree $T$ and for any $r\in V(T)$, we have 
$|M(G)| =  |R(T,r)| =  |R(T,r)| + (t-2)n!$, as desired. 
Now assume that $t\geq 3$, and $T$ is a $t$-vertex tree and $r\in V(T)$. 

First suppose that $r$ is a leaf of $T$. Let $r'$ be the neighbour of $r$ in $T$. 
By \cref{Leaf}, 
$|R(T-r,r')| \leq |R(T,r)| + n!$. 
By induction, 
\begin{align*}
|M(G)| 
 \leq |R(T-r,r')| + (t-3)n! 
& \leq |R(T,r)| +n! + (t-3)n! \\
& = |R(T,r)| + (t-2)n! ,
\end{align*}
 as desired. 

Now assume that $r$ is not a leaf of $T$. So there are subtrees $T_1$ and $T_2$ of $T$ such that $T_1\cup T_2=T$ and $V(T_1\cap T_2)=\{r\}$ and $|V(T_1)|\geq 2$ and $|V(T_2)|\geq 2$. 
By induction, 
\begin{align*}
|M(G)| & \leq |R(T_1,r)| + ( |V(T_1)|-2)n! \text{ and}\\
|M(G)| & \leq |R(T_2,r)| + ( |V(T_2)|-2)n!.
\end{align*}
Thus
\begin{align*}
2|M(G)| 
& \leq  |R(T_1,r)| + |R(T_2,r)| + ( |V(T_1)|+|V(T_2)|-4)n! \\
& =  |R(T_1,r)| + |R(T_2,r)| + ( t-3)n! .
\end{align*}
By \cref{NonLeaf},
\[ |R(T_1,r)| + |R(T_2,r)| \leq |M(G)| + |R(T,r)| + n!. \]
Hence
\begin{align*}
2|M(G)| & \leq |M(G)| + |R(T,r)| + n! + ( t-3)n! .
\end{align*} 
Therefore $|M(G)| \leq |R(T,r)|  + ( t-2)n!$, 
as desired. 
\end{proof}

%\vspace*{-2ex}
%\bibliographystyle{../../BibTeX/DavidNatbibStyle}
%\bibliography{../../BibTeX/myBibliography}
\def\soft#1{\leavevmode\setbox0=\hbox{h}\dimen7=\ht0\advance \dimen7
  by-1ex\relax\if t#1\relax\rlap{\raise.6\dimen7
  \hbox{\kern.3ex\char'47}}#1\relax\else\if T#1\relax
  \rlap{\raise.5\dimen7\hbox{\kern1.3ex\char'47}}#1\relax \else\if
  d#1\relax\rlap{\raise.5\dimen7\hbox{\kern.9ex \char'47}}#1\relax\else\if
  D#1\relax\rlap{\raise.5\dimen7 \hbox{\kern1.4ex\char'47}}#1\relax\else\if
  l#1\relax \rlap{\raise.5\dimen7\hbox{\kern.4ex\char'47}}#1\relax \else\if
  L#1\relax\rlap{\raise.5\dimen7\hbox{\kern.7ex
  \char'47}}#1\relax\else\message{accent \string\soft \space #1 not
  defined!}#1\relax\fi\fi\fi\fi\fi\fi}

\end{document}